\documentclass[11pt]{amsart}

\usepackage[top=1.2in,bottom=1.2in,left=1.2in,right=1.2in,marginpar=1in]{geometry}

\usepackage{latexsym,amsfonts,amssymb,verbatim}
\usepackage{amsmath, latexsym, graphicx}
\usepackage{amsthm}
\usepackage{pinlabel}
\usepackage{hyperref}
\usepackage{tikz}
\usepackage{subcaption}
\usepackage{todonotes}

\newcommand{\nc}{\newcommand}
\nc{\dmo}{\DeclareMathOperator}
\nc{\nt}{\newtheorem}

\nt{theorem}{Theorem}[section]
\nt{lemma}[theorem]{Lemma}
\nt{prop}[theorem]{Proposition}
\nt{question}[theorem]{Question}
\nt{cor}[theorem]{Corollary}
\nt{fact}[theorem]{Fact}
\nt{problem}[theorem]{Problem}
\nt{conjecture}[theorem]{Conjecture}

\theoremstyle{remark}
\nt{remark}[theorem]{Remark}

\theoremstyle{definition}
\nt{definition}[theorem]{Definition}

\nc{\Z}{\mathbb{Z}} 
\nc{\R}{\mathbb{R}} 
\nc{\Q}{\mathbb{Q}} 
\nc{\N}{\mathbb{N}} 

\nc{\M}{\mathcal{M}} 
\nc{\Pnm}{\mathcal{P}} 
\nc{\hg}{\mathcal{H}_g} 
\nc{\hh}{\mathcal{H}_2} 
\nc{\cat}{\operatorname{CAT}(0)} 
\nc{\bd}{\partial} 
\nc{\contact}[1]{\mathcal{C}#1} 
\nc{\fcontact}[1]{\hat{\mathcal{C}}#1} 
\nc{\Dtwo}{\mathcal{D}_2} 
\nc{\ND}{\mathcal{ND}_2} 
\nc{\into}{\hookrightarrow} 
\nc{\mcg}{\operatorname{MCG}} 
\nc{\goeritz}{\mathcal{G}_2} 
\nc{\qieq}[1]{\stackrel{#1}{\asymp}} 
\nc{\base}{\operatorname{base}} 

\nc{\margin}[1]{\marginpar{\tiny #1}}

\nc{\p}[1]{\smallskip\noindent{{\bf #1}}}

\nc{\todoC}[1]{\todo[color=blue!30]{#1}} 
\nc{\todoM}[1]{\todo[color=orange!40]{#1}} 

\begin{document}


\title{Purely Pseudo-Anosov Subgroups of the Genus Two Handlebody Group}

\author{Marissa Chesser and Christopher J Leininger}






\begin{abstract}
We prove that finitely generated, purely pseudo-Anosov subgroups of the genus $2$ handlebody group are convex cocompact.
\end{abstract}

\maketitle

\section{Introduction}\label{sec:intro}

Farb and Mosher \cite{FMcc} defined convex cocompactness for subgroups, $G <\mcg(S)$, of the mapping class group of a closed surface, $S$.  Their work, combined with that of Hamenst\"adt~\cite{hamenstadt}, implies that the associated $\pi_1S$--extension group is hyperbolic if and only if $G$ is convex cocompact, (see also Mj-Sardar \cite{MjSardar}).  A convex cocompact subgroup is necessarily finitely generated and purely pseudo-Anosov, (that is, all infinite order elements are pseudo-Anosov), and as part of Gromov's Hyperbolization Question (see Bestvina's problem list \cite[Question~1.1]{bestvinaproblems}), Farb and Mosher \cite[Question~1.5]{FMcc} asked if the converse is true\footnote{In fact, Farb and Mosher only asked their question for free subgroups, but this more general version is now well-known and often attributed to them.}.

Farb and Mosher's question has now been answered affirmatively for several classes of subgroups; see \cite{kentleiningerschleimer,dowdallkentleininger,KobMangTay,Runnels,benagoeritz,LeiningerRussell-3-mfd}.  In addition, it is also known to hold if one further assumes that the subgroup is undistorted; see \cite{convexcocompact}.  Furthermore, there are now also several alternative characterizations of convex cocompactness of subgroups of mapping class groups; see \cite{kentleiningershadows,hamenstadt,DurhamTaylor}.

To describe our results, we let $V_2$ be a genus $2$ handlebody and write $S = \partial V_2$ to denote its boundary.  The {\em genus $2$ handlebody group} is the subgroup $\hh < \mcg(S)$ consisting of isotopy classes of homeomorphisms of $S$ that extend over $V_2$; see \cite{MasurHandle,HH-handlebody1,HH-handlebody2,chesser-stable}.  Our main result answers Farb and Mosher's question affirmatively for subgroups of $\hh$.

\begin{theorem}\label{thm:main}
    If $G$ is a purely pseudo-Anosov, finitely generated subgroup of $\hh$, then $G$ is convex cocompact in the genus 2 mapping class group.
\end{theorem}

Since the genus 2 Goeritz group $\goeritz$ is a subgroup of $\hh$, Theorem \ref{thm:main} also provides an alternative proof of Tshishiku's theorem, \cite[Theorem~A]{benagoeritz}.
\begin{cor}[Tshishiku]\label{cor:goeritz}
    Finitely generated, purely pseudo-Anosov subgroups of the genus 2 Goeritz group $\goeritz$ are convex cocompact in the genus two mapping class group $\mcg(\bd V_2)$. 
\end{cor}

The assumption in Theorem~\ref{thm:main} that we are in the setting of the genus 2 handlebody group is crucial: we use and expand on the tools developed by Hamenst\"adt and Hensel in \cite{HH-handlebody2}, as well as the work of the first author in \cite{chesser-stable}, which are all specific to the genus 2 case.  See the proof below for details.


\subsection{Acknowledgements}\label{subsec:thanks} The authors would like to thank Jacob Russell for helpful conversations, and Dan Margalit, Sam Taylor, and the anonymous referee for comments on earlier versions of the paper that helped improve the exposition.  The second author was partially supported by NSF DMS-2305286.


\section{Proof of Theorem~\ref{thm:main}}\label{sec:main}

Throughout what follows, we let $S = \bd V_2$.  Any subgroup $G < \mcg(S)$ has a finite index, torsion free subgroup; see Ivanov \cite{Ivanov}.  Passing to a finite index subgroup does not change the property of being convex cocompact, and so to prove Theorem~\ref{thm:main} for $G<\hh$ we may assume, without loss of generality, that $G$ is torsion free.

The proof of Theorem~\ref{thm:main} is roughly divided into three steps.   In \S\ref{S:P and hat P} we recall the construction of a tree $\Pnm$ on which $\hh$ acts, due to Hamenst\"adt and Hensel \cite{HH-handlebody2}, and prove that after adding edges appropriately, we obtain a graph, $\hat{\Pnm}$, quasi-isometric to the disk graph (Proposition~\ref{prop:conediskqi}).  Next, in \S\ref{S:action on P hat P} we analyze the action of a finitely generated, torsion free, purely pseudo-Anosov subgroup $G<\hh$ on $\Pnm$ and $\hat{\Pnm}$, and prove that the orbit map to $\hat{\Pnm}$ is a quasi-isometric embedding (Proposition~\ref{prop:subtreeqi}).  From these two facts, we deduce that the orbit map of such a group $G$ to the disk graph is a quasi-isometric embedding (Corollary~\ref{cor:orbitqie}), and in \S\ref{S:distance formulas} use this, together with the distance formulas of Masur-Minsky and Masur-Schleimer to deduce that $G$ must be quasi-isometrically embedded in the the $\mcg$ (Proposition~\ref{prop:undistorted}).  This immediately proves the Theorem~\ref{thm:main} by appealing to a result of Bestvina, Bromberg, Kent, and the second author \cite{convexcocompact}.

\subsection{A model for the disk graph}\label{S:P and hat P}
In \cite{HH-handlebody2}, Hamenst\"adt and Hensel prove that $\hh$ is a $\rm{CAT}(0)$ group, along the way providing several tools to study the geometry of this group.  One of these tools is the action on a certain full subgraph $\Pnm$ of the {\em pants graph} of $S$: the graph whose vertices are isotopy classes of pants decompositions with edges between those that differ by an {\em elementary move} (see e.g.~\cite{Margalit}).  Specifically, the subgraph $\Pnm$ is spanned by pants decompositions of $S$ consisting of only {\em non-separating meridians}, that is, isotopy classes of non-separating essential curves in $S$ that bound disks in $V_2$.  A key fact we need is the following result of Hamenst\"adt-Hensel \cite[Theorem~5.15]{HH-handlebody2}.

\begin{theorem} \label{T:it's a tree!}
The graph $\Pnm$ is a tree.
\end{theorem}

Given a non-separating meridian $\alpha$, let $\Pnm(\alpha) \subset \Pnm$ denote the subgraph spanned by pants decompositions containing $\alpha$.  The next fact was originally sketched by Hamenst\"adt-Hensel \cite[Corollary~5.18]{MR4255086} and made explicit by the first author in \cite[Lemma~3.1]{chesser-stable}.
\begin{lemma} \label{L:subgraph is a tree}
For any $\alpha$, $\Pnm(\alpha)$ is a subtree of $\Pnm$. \qed
\end{lemma}

We define the {\em electrification}, $\hat{\Pnm}$, to be the graph obtained from $\Pnm$ by, for each non-separating meridian $\alpha$ on $S$, adding a new edge between every pair of vertices in $\Pnm(\alpha)$; in particular, $\Pnm$ and $\hat{\Pnm}$ have the same vertex sets.
The importance of the graph $\hat{\Pnm}$ is seen in the next proposition.  Let $\Dtwo$ denote \emph{disk graph} of $V_2$, ie the subgraph of the curve graph of $S$ spanned by curves that bound a disk in $V_2$.  That is, $\Dtwo$ is the graph whose vertices are meridians, so that two vertices span an edge if and only if the associated meridians have disjoint representatives on $S$. The inclusion of the subgraph, $\ND\subset \Dtwo$ spanned by non-separating meridians was shown to be a quasi-isometry by the first author \cite[Proposition~5.1]{chesser-stable}.

Given a group $H$ acting on a set $X$ and metric space $(Y,d_Y)$, a map $\Psi \colon X \to (Y,d_Y)$ is said to be {\em coarsely $H$--equviariant} if
\[ \sup \{d_Y(h \cdot \Psi(x),\Psi(h \cdot x)) \mid x \in X, h \in H\} < \infty.\]  
That is, $\Psi$ is $H$--equivariant, up to a bounded metric error in $Y$.

\begin{prop}\label{prop:conediskqi}
    The electrification $\hat{\Pnm}$ is coarsely $\hh$-equivariantly quasi-isometric to $\ND$, and hence to $\Dtwo$.
\end{prop}
\begin{proof} 
Construct a map $\Phi \colon \ND^0 \to \Pnm^0 = \hat \Pnm^0$ so that for any vertex $\alpha$, $\Phi(\alpha)$ is some vertex of $\Pnm(\alpha)$.

Given $g \in \hh$ and $\alpha \in \ND^0$, it follows that $\Phi(g \cdot \alpha) \in \Pnm(g \cdot \alpha)^0$. 
Similarly, $g \cdot \Phi(\alpha) \in g \cdot \Pnm(\alpha)^0 = \Pnm(g \cdot \alpha)^0$.  Since both $g \cdot \Phi(\alpha)$ and $\Phi(g \cdot \alpha)$ are in $\Pnm(g \cdot \alpha)^0$, 
these vertices are therefore connected by an edge in $\hat \Pnm$. Thus
\[ \hat d(g \cdot \Phi(\alpha), \Phi(g \cdot \alpha)) \leq 1,\] proving that $\Phi$ is coarsely $\hh$--equivariant.  

Next, suppose $\alpha,\beta \in \ND^0$ are connected by an edge.  This implies that $\alpha,\beta$ are part of pants decomposition $P$ of non-separating disk curves, and hence $P \in \Pnm(\alpha) \cap \Pnm(\beta) \neq \emptyset$.  Thus $\Phi(\alpha)$ and $\Phi(\beta)$ are both connected by an edge of $\hat \Pnm$ to $P$, implying
\[ \hat d(\Phi(\alpha),\Phi(\beta)) \leq 2.\]
It follows that $\Phi$ is $2$--Lipshitz to $\hat \Pnm$.

To prove the lower bound, first note that for any two $P,P' \in \hat \Pnm^0$ connected by an edge of $\hat \Pnm$, there is a curve $\gamma$ in common to both $P$ and $P'$ (regardless of whether the edge is one of the original edges of $\Pnm$ or an added edge in $\hat \Pnm$).
Now suppose $\alpha,\beta \in \ND^0$, and $\hat d(\Phi(\alpha),\Phi(\beta)) = n$.  Connect these points by a geodesic in $\hat \Pnm$ with consecutively ordered indexed vertices
\[ \Phi(\alpha) = P_0,P_1,\ldots,P_n = \Phi(\beta). \]
Because $P_{i-1}$ and $P_i$ are connected by an edge of $\hat \Pnm$, for $i=1,\ldots,n$ there is a curve $\gamma_i$ that is in both of these pants decompositions.  In particular, $\gamma_i$ and $\gamma_{i+1}$ are both in $P_i$, and are therefore disjoint.  Similarly, $\alpha$ and $\gamma_1$ are disjoint, as are $\beta$ and $\gamma_n$.  Thus,
\[ \alpha,\gamma_1,\ldots,\gamma_n,\beta\]
are the consecutively ordered vertices of a path in $\ND$, and hence
\[ d(\alpha,\beta) \leq n+1 = \hat d(\Phi(\alpha),\Phi(\beta)) + 1.\]

Combining the two inequalities, we see that $\Phi$ is a $(2,1)$--quasi-isometric embedding.  Since every vertex $P$ in $\hat \Pnm^0$ is in some $\Pnm(\alpha)^0$, it has distance at most $1$ from $\Phi(\alpha)$. Thus $\Phi$ is coarsely surjective, and therefore a $(2,1)$--quasi-isometry. As already noted, the inclusion of $\ND$ into $\Dtwo$ is an $\hh$--equivariant quasi-isometry, by \cite[Proposition~5.1]{chesser-stable}, and so this completes the proof.
\end{proof}

\subsection{Quasi-isometric orbit map} \label{S:action on P hat P}

We now consider a finitely generated, purely pseudo-Anosov subgroup $G< \hh$.  As mentioned above, for our purposes we need only consider the case that $G$ is torsion free.

\begin{lemma}\label{lem:free}
    Suppose $G <\hh$ is a torsion free, finitely generated, purely pseudo-Anosov subgroup.  Then the action of $G$ on $\Pnm$ is free.
\end{lemma}
\begin{proof}  
Since the vertices of $\Pnm$ are pants decompositions (by disk curves), any element of $\hh$ that fixes a vertex must be reducible. However, since $G$ is torsion free and purely pseudo-Anosov, every element of $G$ must be pseudo-Anosov and thus cannot be reducible. It follows that no nontrivial element fixes a vertex, and so the action of $G$ is free.
\end{proof}

For $G<\hh$ as in the preceding lemma, the action on $\Pnm$ has a minimal invariant subtree $\Pnm_G \subset \Pnm$, which is the union of the axes of all nontrivial elements; see e.g.~Bestvina's survey \cite{Bestvina-Trees}.
From the freeness of this action, we easily deduce the following ({\em cf.}~\cite[Lemma~3.3]{LeiningerRussell-3-mfd}).

\begin{prop} \label{prop:free-qi}
Suppose $G< \hh$ is a torsion free, finitely generated, purely pseudo-Anosov subgroup.  Then $\Pnm_G \subset \Pnm$ is locally finite and $\Pnm_G/G$ is a finite graph.  Furthermore, $G$ is free and any orbit map $G \to \Pnm_G$ is a quasi-isometry.
\end{prop}
\begin{proof}  Taking a base point in $\Pnm_G$ and considering the convex hull of it's translates by a finite set of generators, we obtain a finite subtree of $\Pnm_G$ whose translates cover $\Pnm_G$, and hence projects onto $\Pnm_G/G$, proving this quotient is a finite graph.  The free action is a covering space action and thus $\Pnm_G$ is locally finite and $G$ is free.  Finally, any orbit map $G \to \Pnm_G$ is a quasi-isometry by Milnor-\v{S}varc Lemma; see e.g.~\cite[Ch.~I.8]{bridsonhaefliger}.
\end{proof}

The next proposition relates distances in $G$, or equivalently in $\Pnm_G$, with those in $\hat \Pnm$.
\begin{prop}\label{prop:subtreeqi}
    For any torsion free, finitely generated, purely pseudo-Anosov $G < \hh$, the inclusion $\iota:\Pnm_G\into \hat{\Pnm}$ is a $G$-equivariant quasi-isometric embedding.
\end{prop}

To prove this proposition, we will need to bound the size of $\Pnm_G \cap \Pnm(\alpha)$, independent of $\alpha$.

\begin{lemma}\label{lem:boundedint}
    There exists an $M>0$ such that for any non-separating meridian $\alpha$, the intersection $\Pnm_G\cap \Pnm(\alpha)$ has diameter at most $M$.
\end{lemma}
\begin{proof}  We show that if the diameter of $\Pnm_G \cap \Pnm(\alpha)$ is too large, then $G$ must contain a reducible element.  Toward that end, let $N$ denote the number of vertices of $\Pnm_G/G$, which is the number of $G$--orbits of vertices of $\Pnm_G$.  We will show that $M = 3N$ is a bound for the diameter of $\Pnm_G \cap \Pnm(\alpha)$. 
 We therefore assume the diameter of $\Pnm_G \cap \Pnm(\alpha)$ is at least $3N+1$ for some $\alpha$ and arrive at a contradiction.

By the pigeonhole principle, there is a vertex $P \in \Pnm_G \cap \Pnm(\alpha)$ with $$|G \cdot P \cap \Pnm(\alpha)| \geq 4$$.  Thus there are three distinct, nontrivial elements $g_1,g_2,g_3 \in G$ so that
\[ P,g_1P,g_2P,g_3P \in \Pnm(\alpha).\]
Since $G$ is torsion free and purely pseudo-Anosov, each $g_i$ is pseudo-Anosov.
Writing \\$P = \{\alpha,\beta_1,\beta_2\}$, it follows that
\[ \alpha \in \{g_i(\alpha),g_i(\beta_1),g_i(\beta_2)\},\]
for all $i =1,2,3$. If $g_i(\alpha) = \alpha$ for some $i$, then $g_i$ is reducible, a contradiction.  Therefore, for each $i=1,2,3$, there is a $j(i) \in \{1,2\}$ so that $\alpha = g_i(\beta_{j(i)})$.  By the pigeon-hole principle again, there are $i \neq i'$ so that $j(i) = j(i')$.  Reindexing, we may assume $j(1) = j(2)=1$, so that $g_1(\beta_1) = g_2(\beta_1) = \alpha$, and hence
\[ g_1^{-1}g_2(\beta_1)=g_1^{-1}(\alpha) = \beta_1.\]
Therefore, $g_1^{-1}g_2$ is reducible, Since $g_1 \neq g_2$, $g_1^{-1}g_2$ is a nontrivial element of $G$, hence pseudo-Anosov, which is another contradiction.  

We can thus conclude that the diameter of $\Pnm_G \cap \Pnm(\alpha)$ is less than $3N+1$.
\end{proof}

\begin{proof}[Proof of Proposition~\ref{prop:subtreeqi}.]
Since $\Pnm_G$ is a subgraph, the inclusion is $1$--Lipschitz. It remains to prove the lower bound. Toward this end, suppose $A$ and $B$ are vertices in $\Pnm_G$ connected by a geodesic in $\hat{\Pnm}$ with consecutively ordered indexed vertices
\[ A=P_0,P_1, \dots,P_n=B. \]
For each $i=1,\ldots,n$, let $\Gamma_i$ be a geodesic path in $\Pnm$ between $P_{i-1}$ and $P_i$. Additionally, let $\gamma_i$ be a curve that is in common between the pants decompositions $P_{i-1}$ and $P_i$, which must exist since $P_{i-1}$ and $P_i$ are adjacent in $\hat{\Pnm}$. Notice that since $\gamma_i\in P_{i-1} \cap P_i$, we have $P_{i-1},P_i \in \Pnm^{0}(\gamma_i)$. 
 It follows that each geodesic $\Gamma_i$ must be entirely contained in $\Pnm(\gamma_i)$, since $\Pnm(\gamma_i)$ is a subtree of $\Pnm$, by Lemma~\ref{L:subgraph is a tree}.

Consider the path $\Gamma_1\cup \Gamma_2$. Since $\Pnm$ is a tree, this is either already a geodesic between $P_0$ and $P_2$, or there is some backtracking. Let $L_1$ be the point in $\Gamma_1 \cap \Gamma_2$ that is furthest from $P_1$ in $\Pnm$, or let $L_1=P_1$ in the case that $\Gamma_1\cup \Gamma_2$ is already a geodesic. Notice that since 
\[  L_1\in \Gamma_1\cap \Gamma_2 \subset \Pnm(\gamma_1)\cap \Pnm(\gamma_2)\]
and since $P_0\in \Pnm(\gamma_1)$ and $P_2\in\Pnm(\gamma_2)$, it follows that there are edges in $\hat{\Pnm}$ from $P_0$ to $L_1$ and from $L_1$ to $P_2$. Additionally, if $\Gamma_1'$ is the geodesic path in $\Pnm$ from $P_0$ to $L_1$ and $\Gamma_2'$ is the geodesic path from $L_1$ to $P_2$, then by construction $\Gamma_1'\cup \Gamma_2'$ has no backtracking and is thus a geodesic in $\Pnm$ from $P_0$ to $P_2$. 

For each subsequent $P_i$ with $2 \leq i \leq n-1$, we proceed similarly: replacing $P_i$ with the vertex $L_i\in \Gamma_i'\cap \Gamma_{i+1}$ that is furthest from $P_i$, replacing $\Gamma_i'$ with the geodesic $\Gamma_i''$ in $\Pnm$ from $L_{i-1}$ to $L_i$, and replacing $\Gamma_{i+1}$ with the geodesic $\Gamma_{i+1}'$ in $\Pnm$ from $L_i$ to $P_{i+1}$. We note that there can never be any overlap between any $\Gamma_j$ and $\Gamma_k'$ for $j-k \geq2$ because otherwise there would be some vertex adjacent to both $P_{k-1}$ and $P_j$, implying $j-(k-1) = d(P_{k-1},P_j) \leq 2$, which is a contradiction.

In this manner, we have constructed a geodesic 
\[ \Gamma= \Gamma_1' \cup \Gamma_2'' \cup \Gamma_3''\cup \cdots \cup \Gamma_{n-1}'' \cup\Gamma_n'\]
in $\Pnm$ between $A$ and $B$, and replaced the original geodesic in $\hat{\Pnm}$ between $A$ and $B$ with a new one consisting of consecutively ordered vertices
\[ P_0, L_1,L_2,\cdots, L_{n-1}, P_n.\]
Because $\Gamma$ is a geodesic in $\Pnm$ between two points in $\Pnm_G$, and because $\Pnm_G$ is a subtree of $\Pnm$, $\Gamma$ is a geodesic remaining entirely inside $\Pnm_G$. Moreover, by Lemma \ref{lem:boundedint}, each segment $\Gamma_1',\Gamma_2'',\dots,\Gamma_{n-1}'', \Gamma_n'$ has length at most $M$. Therefore, $\frac{1}{M}d_{\Pnm_G}(A,B)\leq \hat{d}(A,B)$, and thus the inclusion $\iota$ is a quasi-isometric embedding.

Because $\iota$ is an inclusion, for any $g\in G$, $g\cdot \iota(A)=g\cdot A = \iota(g\cdot A)$. Thus, $\iota$ is $G$-equivariant. This completes the proof.
\end{proof}

\begin{cor}\label{cor:orbitqie}
    If $G< \hh$ is purely pseudo-Anosov and finitely generated, then the orbit map of $G$ into $\Dtwo$ is a quasi-isometric embedding.
\end{cor}
\begin{proof}
As noted above, we may assume without loss of generality that $G$ is torsion free. By Proposition~\ref{prop:free-qi}, any orbit map from $G$ to $\Pnm_G$ is a quasi-isometry, while by Proposition~\ref{prop:subtreeqi}, the inclusion $\Pnm_G\into \hat{\Pnm}$ is a 
 $G$-equivariant quasi-isometric embedding. Combining this with the fact that $\hat{\Pnm}$ is coarsely $G$-equivariantly quasi-isometric to $\Dtwo$, by Proposition~\ref{prop:conediskqi}, we see that any orbit map of $G$ to $\Dtwo$ is a quasi-isometric embedding.
\end{proof}
We note that combining this corollary with \cite[Theorem~1.1]{chesser-stable} also gives us the following corollary.
\begin{cor}\label{cor:stable}
    If $G\leq \hh$ is purely pseudo-Anosov and finitely generated, then $G$ is stable in $\hh$.
\end{cor}

\subsection{Distance formulas and convex cocompactness} \label{S:distance formulas}

The next proposition provides the final ingredient for the proof of Theorem~\ref{thm:main}.
\begin{prop}\label{prop:undistorted}
    If $G< \hh$ is purely pseudo-Anosov and finitely generated, then $G$ is quasi-isometrically embedded in $\mcg(S)$.
\end{prop}

The proof of this proposition relies on the distance formulas for the mapping class group and the disk graph. The mapping class group distance formula was discovered by Masur and Minsky \cite{MR1791145}, while the formula for the disk graph was proved by Masur and Schleimer \cite{MR2983005}. In order to state these, we will briefly recall the relevant ideas and terminology, but refer the reader to \cite{MR1791145} for full details.

Given a (multi)curve $\alpha$ and subsurface $Y \subset S$, $\pi_Y(\alpha)$ denotes the subsurface projection of $\alpha$ to the curve graph of $Y$.  For curves $\alpha$ and $\beta$, $\pi_{\beta}(\alpha)$ will refer to the projection of $\alpha$ to the curve graph of an annulus with core curve $\beta$.  All curves and subsurfaces are considered up to isotopy, and representatives of the isotopy classes are assumed to be in minimal position.

Given a curve $\alpha$ on $S$, a \emph{clean transverse curve} for $\alpha$ is a curve $\beta$ in $S$ whose regular neighborhood is either a $1$-holed torus or a $4$-holed sphere. A \emph{complete, clean marking} $\mu$ is a set of pairs $\{ (\alpha_i, t_i)\}$ such that $\operatorname{base}(\mu)=\{\alpha_i\}$ forms a maximal simplex in the curve graph (i.e.~$\{\alpha_i\}$ is a pants decomposition) called the {\em base of $\mu$}. For each $i$, $t_i = \pi_{\alpha_i}(\beta_i)$ for some clean transverse curve $\beta_i$ for $\alpha_i$, and $\operatorname{trans}(\mu) = \{t_i\}$ are called the {\em transversals of $\mu$}.
The \emph{marking graph} of $S$, denoted $\widetilde{\mathcal{M}}$, is a graph whose vertices represent complete clean markings on $S$ and for which edges represent two complete clean markings differing by an ``elementary move''; again, see Masur and Minsky \cite{MR1791145} for precise definitions.

We recall that the projection of a complete, clean marking $\mu$ to a subsurface $Y\subset S$ is defined as follows. When $Y$ is an annulus whose core curve is some $\alpha_i\in\operatorname{base}(\mu)$, then $\pi_Y(\mu) = t_i = \pi_Y(\beta_i)$. Otherwise, $\pi_Y(\mu) = \pi_Y(\operatorname{base}(\mu))$.  Given multicurves or markings $\mu,\mu'$ and a subsurface $X$ of $S$, we write
\[ d_X(\mu,\mu') = \operatorname{diam}(\pi_X(\mu),\pi_X(\mu')),\]
where the diameter is computed in the curve graph of $X$.

It is convenient to use two more pieces of notation. Given non-negative real numbers $A,B,C$ with $A \geq 1$, we write $B \qieq{A} C$ to mean $\frac{1}{A}B - A \leq C \leq AB+A$, and further define $[B]_A = B$ if $B \geq A$, and $[B]_A = 0$ otherwise.  Masur and Minsky's distance formula \cite{MR1791145} states that there is some $C_0 > 0$ such that for any $c\geq C_0$, there is some $A_0\geq 1$ such that
\[ d_{\widetilde{\mathcal{M}}}(\mu,\mu') \qieq{A_0} \sum_{X\in \mathbf{X}} [d_X(\mu,\mu')]_c,\]
where $\mathbf{X}$ is the set of all essential subsurfaces of $S$. The orbit map from $\mcg(S)$ to $\widetilde{\mathcal{M}}$ is a quasi-isometry (see \cite{MR1791145}), and thus for any $\mu \in \widetilde{\mathcal{M}}$ there is some constant $C_0$ such that for any $c\geq C_0$, there is a constant $A_0\geq 1$ such that for all $g,h \in \mcg(S)$,
\begin{equation} \label{E:mcg distance} d_{\mcg(S)}(g,h) \qieq{A_0} \sum_{X\in \mathbf{X}} [d_X(g \cdot \mu,h \cdot \mu)]_c.
\end{equation}


Masur and Schleimer \cite{MR2983005} proved an analogous distance formula for the disk graph. Specifically, there is some constant $C_1$ such that for any $c\geq C_1$, there is a constant $A_1\geq 1$ such that for all $\alpha,\beta \in \Dtwo$,
\begin{equation} \label{E:disk distance} d_{\Dtwo}(\alpha,\beta) \qieq{A_1} \sum_{X\in \mathbf{W}} [d_X(\alpha,\beta)]_c.
\end{equation}
Here $\mathbf{W} \subset \mathbf{X}$ denotes the set of \textit{witnesses} (also called \textit{holes}) for the disk graph. These are essential subsurfaces $X \subset S$ such that every representative of every meridian on $V_2$ has non-empty intersection with $X$. 

\begin{proof}[Proof of Proposition \ref{prop:undistorted}]
First, $G\leq \hh\leq \mcg(S)$, therefore $d_{\mcg(S)}(1,g)\leq d_G(1,g)$ for any $g\in G$ (assuming the generating set for $G$ is contained in that of $\mcg(S)$). It remains to bound $d_{\mcg(S)}(1,g)$ from below by a linear function of $d_G(1,g)$. For this, we use the distance formulas for the mapping class group and the disk graph, together with Corollary \ref{cor:orbitqie}.

By Corollary~\ref{cor:orbitqie}, for any $\alpha \in \Dtwo^0$, there exists $K\geq 1, C\geq 0$ such that 
\begin{equation} \label{E:group disk} \frac{1}{K}d_G(1,g) - C \leq d_{\Dtwo} (\alpha, g\cdot \alpha), \end{equation}
for all $g \in G$.  Now let $P$ be a pants decomposition of meridians, including $\alpha$, and choose a marking $\mu$ so that $\base(\mu) = P$.  Let $C_0$ and $C_1$ be the constants required for \eqref{E:mcg distance} and \eqref{E:disk distance}, respectively, let $c > \max\{C_0,C_1\}$; and let $A_0$ and $A_1$ be such that \eqref{E:mcg distance} and \eqref{E:disk distance} both hold.  Now, we observe that
\begin{align*}
    d_{\Dtwo} (\alpha,g\cdot \alpha) &\leq A_1 \sum_{X\in \mathbf{W}} [d_X(\alpha, g\cdot \alpha)]_c +A_1 \\
    & \leq A_1 \sum_{X \in \mathbf{W}} [d_X(\mu,g \cdot \mu)]_c + A_1 \leq A_1\sum_{X \in \mathbf{X}} [d_X(\mu,g \cdot \mu)]_c + A_1.
\end{align*}
The first inequality is from \eqref{E:disk distance}, while the second comes from the fact that $\alpha \in \base(\mu)$, (meridians are not cores of annular witnesses, so the transversals play no role when summing over $X \in \mathbf{W}$).  The third inequality follows since $\mathbf{W} \subset \mathbf{X}$.

Combining this inequality with \eqref{E:mcg distance}, we have
\[ d_{\Dtwo}(\alpha,g \cdot \alpha) \leq A_1(A_0 d_{\mcg(S)}(1,g) + A_0) + A_1.\]
Combined with \eqref{E:group disk}, we have
\[ \frac{1}{K A_0A_1} d_G(1,g) - \left(\frac{C+A_0A_1+A_1}{A_0A_1} \right) \leq d_{\mcg(S)}(1,g), \]
which is the required lower bound.   
\end{proof}

We finally prove the main theorem by applying the preceding proposition and a theorem of Bestvina, Bromberg, Kent, and the second author \cite{convexcocompact}.
\begin{proof}[Proof of Theorem \ref{thm:main}]
    Suppose we are given a finitely generated, purely pseudo-Anosov subgroup $G <\hh$. After passing to a finite index subgroup if necessary, we can assume that $G$ is torsion free.  By Proposition \ref{prop:undistorted}, $G$ is undistorted in $\mcg(S)$. By the Main Theorem of \cite{convexcocompact}, it follows that $G$ is convex cocompact.
\end{proof}

We note that the genus $2$ assumption was used to deduce Corollary~\ref{cor:orbitqie}, but from that point forward all the results we used in the proof are true for arbitrary genus $g$.  Specifically, the distance formulas from \cite{MR1791145,MR2983005} and the main result of \cite{convexcocompact} are valid for all genus $g \geq 2$.   We can thus hypothesize the conclusion of that Corollary~\ref{cor:orbitqie}, and the remainder of the proof goes through verbatim.  We include this statement as it may be of independent interest.

\begin{theorem} \label{T:general g}
Suppose $G < {\mathcal H}_g$ is a finitely generated, purely pseudo-Anosov subgroup of the genus $g$ handlebody group for any $g \geq 2$.  If the orbit map of $G$ to the disk graph ${\mathcal D}_g$ is a quasi-isometric embedding, then $G$ is convex cocompact. \qed
\end{theorem}

\bibliographystyle{amsalpha}
\bibliography{bib}

\providecommand{\bysame}{\leavevmode\hbox to3em{\hrulefill}\thinspace}
\providecommand{\MR}{\relax\ifhmode\unskip\space\fi MR }
\providecommand{\MRhref}[2]{%
  \href{http://www.ams.org/mathscinet-getitem?mr=#1}{#2}
}
\providecommand{\href}[2]{#2}
\begin{thebibliography}{BBKL20}

\bibitem[BBKL20]{convexcocompact}
Mladen Bestvina, Kenneth Bromberg, Autumn~E. Kent, and Christopher~J.
  Leininger, \emph{Undistorted purely pseudo-{A}nosov groups}, J. Reine Angew.
  Math. \textbf{760} (2020), 213--227. \MR{4069890}

\bibitem[Besa]{bestvinaproblems}
M.~Bestvina, \emph{{Questions in geometric group theory}},
  \texttt{www.math.utah.edu/$\sim$bestvina/}.

\bibitem[Besb]{Bestvina-Trees}
Mladen Bestvina, \emph{{Real trees in topology, geometry, and group theory,
  Preprint}}, \texttt{arXiv:math/9712210},.

\bibitem[BH99]{bridsonhaefliger}
Martin~R. Bridson and Andr{\'e} Haefliger, \emph{Metric spaces of non-positive
  curvature}, Grundlehren der Mathematischen Wissenschaften [Fundamental
  Principles of Mathematical Sciences], vol. 319, Springer-Verlag, Berlin,
  1999. \MR{1744486 (2000k:53038)}

\bibitem[Che22]{chesser-stable}
Marissa Chesser, \emph{Stable subgroups of the genus 2 handlebody group},
  Algebr. Geom. Topol. \textbf{22} (2022), no.~2, 919--971. \MR{4464468}

\bibitem[DKL14]{dowdallkentleininger}
Spencer Dowdall, Autumn~E. Kent, and Christopher~J. Leininger,
  \emph{Pseudo-{A}nosov subgroups of fibered 3-manifold groups}, Groups Geom.
  Dyn. \textbf{8} (2014), no.~4, 1247--1282. \MR{3314946}

\bibitem[DT15]{DurhamTaylor}
Matthew~Gentry Durham and Samuel~J. Taylor, \emph{Convex cocompactness and
  stability in mapping class groups}, Algebr. Geom. Topol. \textbf{15} (2015),
  no.~5, 2839--2859. \MR{3426695}

\bibitem[FM02]{FMcc}
B.~Farb and L.~Mosher, \emph{Convex cocompact subgroups of mapping class
  groups}, Geom. Topol. \textbf{6} (2002), 91--152 (electronic). \MR{MR1914566
  (2003i:20069)}

\bibitem[Ham]{hamenstadt}
Ursula Hamenst{\"a}dt, \emph{{Word hyperbolic extensions of surface groups}},
  Preprint, \texttt{arXiv:math.GT/0505244}.

\bibitem[HH12]{HH-handlebody1}
Ursula Hamenst\"{a}dt and Sebastian Hensel, \emph{The geometry of the
  handlebody groups {I}: distortion}, J. Topol. Anal. \textbf{4} (2012), no.~1,
  71--97. \MR{2914874}

\bibitem[HH21a]{HH-handlebody2}
\bysame, \emph{The geometry of the handlebody groups {II}: {D}ehn functions},
  Michigan Math. J. \textbf{70} (2021), no.~1, 23--53. \MR{4255086}

\bibitem[HH21b]{MR4255086}
\bysame, \emph{The geometry of the handlebody groups {II}: {D}ehn functions},
  Michigan Math. J. \textbf{70} (2021), no.~1, 23--53. \MR{4255086}

\bibitem[Iva92]{Ivanov}
Nikolai~V. Ivanov, \emph{Subgroups of {T}eichm\"{u}ller modular groups},
  Translations of Mathematical Monographs, vol. 115, American Mathematical
  Society, Providence, RI, 1992, Translated from the Russian by E. J. F.
  Primrose and revised by the author. \MR{1195787}

\bibitem[KL08]{kentleiningershadows}
Autumn~E. Kent and Christopher~J. Leininger, \emph{Shadows of mapping class
  groups: capturing convex cocompactness}, Geom. Funct. Anal. \textbf{18}
  (2008), no.~4, 1270--1325. \MR{2465691 (2009j:20056)}

\bibitem[KLS09]{kentleiningerschleimer}
Autumn~E. Kent, Christopher~J. Leininger, and Saul Schleimer, \emph{Trees and
  mapping class groups}, J. Reine Angew. Math. \textbf{637} (2009), 1--21.
  \MR{2599078 (2011a:57002)}

\bibitem[KMT17]{KobMangTay}
Thomas Koberda, Johanna Mangahas, and Samuel~J. Taylor, \emph{The geometry of
  purely loxodromic subgroups of right-angled {A}rtin groups}, Trans. Amer.
  Math. Soc. \textbf{369} (2017), no.~11, 8179--8208. \MR{3695858}

\bibitem[LR22]{LeiningerRussell-3-mfd}
Christopher~J Leininger and Jacob Russell, \emph{Pseudo-{A}nosov subgroups of
  general fibered 3-manifold groups}, Preprint, \texttt{arXiv:2204.04111}, to
  appear {\em {T}rans.~{A}mer.~{M}ath.~{S}oc.}, 2022.

\bibitem[Mar04]{Margalit}
Dan Margalit, \emph{Automorphisms of the pants complex}, Duke Math. J.
  \textbf{121} (2004), no.~3, 457--479. \MR{2040283}

\bibitem[Mas86]{MasurHandle}
Howard Masur, \emph{Measured foliations and handlebodies}, Ergodic Theory
  Dynam. Systems \textbf{6} (1986), no.~1, 99--116. \MR{837978}

\bibitem[MM00]{MR1791145}
H.~A. Masur and Y.~N. Minsky, \emph{Geometry of the complex of curves. {II}.
  {H}ierarchical structure}, Geom. Funct. Anal. \textbf{10} (2000), no.~4,
  902--974. \MR{1791145}

\bibitem[MS12]{MjSardar}
Mahan Mj and Pranab Sardar, \emph{A combination theorem for metric bundles},
  Geom. Funct. Anal. \textbf{22} (2012), no.~6, 1636--1707. \MR{3000500}

\bibitem[MS13]{MR2983005}
Howard Masur and Saul Schleimer, \emph{The geometry of the disk complex}, J.
  Amer. Math. Soc. \textbf{26} (2013), no.~1, 1--62. \MR{2983005}

\bibitem[Run21]{Runnels}
Ian Runnels, \emph{Effective generation of right-angled {A}rtin groups in
  mapping class groups}, Geom. Dedicata \textbf{214} (2021), 277--294.
  \MR{4308279}

\bibitem[Tsh21]{benagoeritz}
Bena Tshishiku, \emph{Convex-cocompact subgroups of the goeritz group}, 2021.

\end{thebibliography}

\end{document}